\documentclass[a4paper,twoside,11pt]{article}
\usepackage{amsmath,latexsym,amssymb,amsfonts,amsbsy,amsthm,mathrsfs}
\usepackage{algorithm}
\usepackage{algorithmic}
\usepackage{cite}
\usepackage{array}
\usepackage{tikz}
\usepackage{float}
\usepackage{booktabs}
\usepackage{multirow}
\usetikzlibrary{calc}
\usepackage{indentfirst}
\usepackage{makecell}

\usepackage[labelfont=bf, labelsep=period, figurename=Fig.]{caption}

\usepackage{graphicx}
\usepackage{subfigure}
\usepackage[top=1in, bottom=1in, left=1.25in, right=1.25in]{geometry}
\usepackage{epstopdf}
\usepackage{diagbox}
\usepackage{multirow}
\usepackage{enumerate}
\usepackage{color}
\usepackage{hyperref}
\usepackage{cleveref}
\newfont{\bb}{msbm10}

\newtheorem{theorem}{Theorem}[section]

\newtheorem{lemma}{Lemma}[section]

\newtheorem{remark}{Remark}[section]

\theoremstyle{definition}
\newtheorem{example}{Example}[section]

\numberwithin{equation}{section}

\renewcommand{\arraystretch}{1.24}

\title{Randomized average block coordinate descent method with heavy-ball momentum for tensor least squares problem under the t-product}

\author{
Li-Lin Ji\\
School of Mathematical Sciences, Tongji University, \\
Shanghai 200092, China \\
Email: 2251237@tongji.edu.cn,\\
Ni-Hong Ke\\
School of Mathematical Sciences, Tongji University, \\
Shanghai 200092, China \\
Email: kenihong@tongji.edu.cn,\\
Jun-Feng Yin\\
Key Laboratory of Intelligent Computing and
Applications (Tongji University),
\\Ministry of Education,
School of Mathematical Sciences, Tongji University, \\
Shanghai 200092, China \\
Email: yinjf@tongji.edu.cn\\
}

\begin{document}
\cleardoublepage \pagestyle{myheadings}

\maketitle
\date{}

\markboth{\small}
{\small  }

\begin{abstract}
A tensor randomized average block coordinate descent method with heavy-ball momentum is proposed for the tensor least squares problem with respect to the t-product.
Theoretical analysis for tensor block coordinate descent method is established and average techniques are applied to avoid the computation of tensor Moore--Penrose inverse. To further accelerate convergence, a heavy-ball momentum scheme is incorporated into the block coordinate descent method, where the step size and momentum parameters are adaptively determined by a two-dimensional minimal residual projection. Theoretical analyses give the convergence of the new method and provide an improved bound on the linear convergence rate. Numerical experiments further verify the efficiency of the proposed method in terms of the number of iterations and the better video recovery performance.

\end{abstract}


\bigskip

\section{Introduction}
\label{sec:intro}

Consider solving the tensor linear system
\begin{equation}\label{eqn:taxb}
	\mathcal{A}*\mathcal{X} = \mathcal{B}, \quad \mathcal{A} \in \mathbb{R}^{n_1 \times n_2 \times n_3}, \quad\mathcal{B} \in \mathbb{R}^{n_1 \times p \times n_3},
\end{equation}
where ``$*$" is the tensor t-product proposed by Kilmer and Martin \cite{2011KM}. Tensor linear systems under the t-product framework widely arise in various scientific computing and artificial intelligence applications, such as computer vision \cite{2013NKBH,2018YGXG}, image processing \cite{2013KKNH, 2013MSL,2026JI}, data completion \cite{2022HLX, 2023YZ}
and machine learning \cite{2009KB}.

In practical applications, noise, incomplete observations and modeling errors often lead to the following tensor least squares problem:
\begin{equation}\label{eqn:tls}
	\min\limits_{\mathcal{X}}\|\mathcal{B}-\mathcal{A}*\mathcal{X}\|_F^2,
\end{equation}
where $\|\mathcal{A}\|_F^2=\sum\limits_{i,j,k}\mathcal{A}_{i,j,k}\mathcal{A}_{i,j,k}$ denotes the Frobenius norm of $\mathcal{A}$.
Needell, Srebro and Ward investigated the relations among the randomized Kaczmarz method, random sampling and stochastic gradient method and presented the convergence properties for tensor linear equations and tensor least squares problems \cite{2016NSW}.
Huang and Zhong proposed a tensor randomized extended Kaczmarz method, as well as the block and greedy variants, and and proved convergence in expectation to the least squares solution $\mathcal{X}_{LS}$ \cite{2024HZ}.
An, Liang, et al. \cite{2025ALJL} further developed a tensor randomized extended average block Kaczmarz method, which employs convex linear combinations of tensors to avoid computing the tensor pseudoinverse.
Tang and Li recently studied statistical properties of random sampling methods for tensor least squares problems, including leverage-score subsampling approaches \cite{2025TL}.

Recently, randomized coordinate descent methods for linear least squares have recently attracted considerable attention in the numerical linear algebra and optimization communities.
Leventhal and Lewis~\cite{2010LL} introduced a randomized coordinate descent method that selects a coordinate direction at random and established a convergence rate characterized by the scaled condition number. 
Generalizations were subsequently studied by Nesterov~\cite{2012N}, while Beck and Tetruashvili~\cite{2013BT} developed a unified convergence theory for randomized coordinate descent method. 
More research on randomized block coordinate descent methods was then studied by Richt\'{a}rik and Tak\'{a}\v{c}~\cite{2014RT}, Lu and Xiao~\cite{2015LX} and Necoara et al.~\cite{2017NNG}.

For the tensor least squares problem~\eqref{eqn:tls}, a randomized coordinate descent method was proposed in~\cite{2026KJ}, where one lateral slice of $\mathcal{A}$ is randomly selected per iteration and an upper bound of the convergence rate is established. Building on this foundation, we further propose and analyze the tensor randomized block coordinate descent method and its average block variants.
Moreover, a tensor randomized average block coordinate descent method with heavy-ball momentum is proposed for the tensor least squares problem where both the step size and the momentum parameters are adaptively determined by the minimal residual projection principle. Theoretical analysis proves that the new method
converges to the least squares solution $\mathcal{X}_{LS}$ at a linear convergence rate. 
Numerical experiments further verify that the proposed method is efficient, and better than the existing tensor solvers in terms of the number of iteration steps and computational time.

The rest of the paper is organized as follows. Preliminaries of notations are presented in Section \ref{sec:pre}. In Section \ref{sec:trabcd}, the tensor randomized average block coordinate descent method with heavy-ball momentum is proposed and the convergence theory is established.
The numerical experiments are given in Section \ref{sec:num} to demonstrate the effectiveness of the proposed method in solving tensor least squares problem. Finally we conclude the paper with a brief summary in Section \ref{sec:con}.

\section{Tensor randomized average block coordinate descent method}\label{sec:pre}

In this section, we first present the tensor randomized average block coordinate descent method, which will serve as the baseline scheme for the heavy-ball momentum acceleration developed later.


The tensor randomized coordinate descent method was proposed in~\cite{2026KJ} as an effective randomized solver for tensor least squares problems. Based on the decomposition
\[
\|\mathcal{B}-\mathcal{A}*\mathcal{X}\|_F^2
=
\left\|
\mathcal{B}
-
\sum_{j=1}^{n_2}
\mathcal{A}_{:,j,:}*\mathcal{X}_{j,:,:}
\right\|_F^2,
\]
this method randomly selects a lateral slice $\mathcal{A}_{:,j_k,:}$ at each iteration and updates the corresponding horizontal slice $\mathcal{X}_{j_k,:,:}^{(k)}$ by minimizing the objective function below:
\begin{equation*}
	\begin{aligned}
		\mathcal{X}^{(k+1)}_{j_k, :, :}
		&= \mathop{\arg} \mathop{\min}_{\mathcal{X} \in \mathbb{R}^{1 \times p \times n_3}} \left\| \mathcal{B} - \sum_{j \in [n_2] \setminus \{j_k\}} \mathcal{A}_{:,j,:} * \mathcal{X}^{(k)}_{j, :, :} - \mathcal{A}_{:,j_k,:} * \mathcal{X} \right\|_F^2 \\
		&= \mathop{\arg} \mathop{\min}_{\mathcal{X} \in \mathbb{R}^{1 \times p \times n_3}} \left\| \mathcal{R}^{(k)} - \mathcal{A}_{:,j_k,:} * \left(\mathcal{X} - \mathcal{X}^{(k)}_{j_k, :, :}\right) \right\|_F^2.
	\end{aligned}
\end{equation*}
So, the update is given by
\begin{equation}\label{eqn:rcd}
\mathcal{X}_{j_k,:,:}^{(k+1)}
=\mathcal{X}_{j_k,:,:}^{(k)}
+
(\mathcal{A}_{:,j_k,:}^{T}*\mathcal{A}_{:,j_k,:})^{\dagger}\mathcal{A}_{:,j_k,:}^{T}\mathcal{R}^{(k)}=
\mathcal{X}_{j_k,:,:}^{(k)}
+
\mathcal{A}_{:,j_k,:}^{\dagger}*\mathcal{R}^{(k)},
\end{equation}
and the residual is updated as
\[
\mathcal{R}^{(k+1)}
=
\mathcal{R}^{(k)}
-
\mathcal{A}_{:,j_k,:}*
\left(
\mathcal{X}_{j_k,:,:}^{(k+1)}
-
\mathcal{X}_{j_k,:,:}^{(k)}
\right).
\]

To improve the efficiency of the tensor randomized coordinate descent method, a natural extension is to employ blockwise updates. In the development of matrix block iteration methods, block selection strategies have gradually evolved from relatively simple randomized rules to more adaptive greedy mechanisms. Early studies focused on randomized average block strategies, in which the working block is selected according to prescribed random sampling rules \cite{2014NT,2015NZZ}. To better leverage the residual information at each iteration, randomized greedy block strategies were subsequently developed, enabling more adaptive block selection and accelerating practical convergence \cite{2020NZ,2022MW,2023XYZ}.

In this paper, we adopt a randomized average blocking strategy. Specifically, let $\mathcal{T}=\{\tau_j\}_{j=1}^q$ be a partition of $[n_2]$ into $q$ disjoint blocks, satisfying
\[
\tau_i\cap\tau_j=\emptyset,\quad i\neq j,
\qquad \text{and} \qquad
\bigcup_{j=1}^q \tau_j=[n_2].
\] At the $(k+1)$-th iteration, one block $\mathcal{A}_{:,\tau_{j_k},:}$ is randomly selected, and only the corresponding horizontal slices $\mathcal{X}^{(k)}_{\tau_{j_k},:,:}$ of the current iterate $\mathcal{X}^{(k)}$ are updated as follows:
\begin{equation}\label{eqn:tRBCD-update}
	\mathcal{X}_{\tau_{j_k},:,:}^{(k+1)}
	=\mathcal{X}_{\tau_{j_k},:,:}^{(k)}
	+
	(\mathcal{A}_{:,\tau_{j_k},:}^{T}*\mathcal{A}_{:,\tau_{j_k},:})^{\dagger}*\mathcal{A}_{:,\tau_{j_k},:}^{T}*\mathcal{R}^{(k)}=
	\mathcal{X}_{\tau_{j_k},:,:}^{(k)}
	+
	\mathcal{A}_{:,\tau_{j_k},:}^{\dagger}*\mathcal{R}^{(k)}.
\end{equation}
This update rule gives rise to the tensor randomized block coordinate descent (tRBCD) method. The details of this method are summarized in Algorithm \ref{alg:tRBCD}, while its convergence analysis is provided in Theorem \ref{thm:TRBCD}.

\begin{algorithm}[!htbp]
	\caption{The tensor randomized block coordinate descent (tRBCD) method}
	\begin{algorithmic}[1]
		\REQUIRE $\mathcal{X}^{(0)} \in \mathbb{R}^{n_2 \times p \times n_3}$,
		$\mathcal{A} \in \mathbb{R}^{n_1 \times n_2 \times n_3}$,
		$\mathcal{B} \in \mathbb{R}^{n_1 \times p \times n_3}$,
		a block partition of $[n_2]$: $\mathcal{T}=\{\tau_j\}_{j=1}^q$.
		\STATE $\mathcal{R}^{(0)} = \mathcal{B}-\mathcal{A}*\mathcal{X}^{(0)}$.
		\FOR{$k = 0, 1, 2, \ldots$}
		\STATE Select $\tau_{j_k} \in \mathcal{T}$ randomly with probability
		$\|\mathcal{A}_{:,\tau_{j_k},:}\|_F^2/\|\mathcal{A}\|_F^2$.
		\STATE
		$
		\mathcal{X}^{(k+1)}_{\tau_{j_k},:,:} = \mathcal{X}^{(k)}_{\tau_{j_k},:,:}
		+(\mathcal{A}_{:,\tau_{j_k},:})^{\dagger}*\mathcal{R}^{(k)}.
		$
		\STATE $\mathcal{R}^{(k+1)}=\mathcal{R}^{(k)}-\mathcal{A}_{:,\tau_{j_k},:}
		*(\mathcal{X}^{(k+1)}_{\tau_{j_k},:,:}-\mathcal{X}^{(k)}_{\tau_{j_k},:,:})$.
		\ENDFOR
	\end{algorithmic}\label{alg:tRBCD}
\end{algorithm}

\begin{theorem}\label{thm:TRBCD}
	The iterative sequence $\{\mathcal{X}^{(k)}\}_{t=0}^{\infty}$ generated by the tensor randomized block coordinate descent method satisfies
	\begin{equation}\label{eqn:tRBCD}
		\mathbb{E}\|\mathcal{E}^{(k+1)}\|_F^2\le \left(1-\min\limits_{\tau_j\in \mathcal{T}}\sigma_{\min }\left(\mathbb{E}\left[\operatorname{bcirc}(\mathcal{A}_{:,\tau_{j},:}*(\mathcal{A}_{:,\tau_{j},:})^{\dagger})\right]\right)\right)^{k+1} \|\mathcal{E}^{0}\|_F^2,
	\end{equation}
	where $
	\mathcal{E}^{(k)}
	:=
	\mathcal{A} * \left(\mathcal{X}_{LS}-\mathcal{X}^{(k)}\right),
	$ and $\mathcal{X}_{LS}$ is the least squares solution of system \eqref{eqn:taxb}.
\end{theorem}

It is worth noting that the update~\eqref{eqn:tRBCD-update} requires computing the Moore--Penrose inverse $\bigl(\mathcal{A}_{:,\tau_{j_k},:}\bigr)^\dagger$ at each iteration, which may become the dominant computational cost when the block size is large.

To reduce the per-iteration complexity, we adopt the average projection technique of \cite{2019N, 2022BL}, replacing the pseudoinverse step in Eq.~\eqref{eqn:tRBCD-update} with a gradient-based approximation. This leads to the following iterative scheme for the tensor randomized average block coordinate descent method:
\[
\mathcal{X}^{(k+1)}_{\tau_{j_k},:,:} =
\mathcal{X}^{(k)}_{\tau_{j_k},:,:}
+\alpha_k
\frac{(\mathcal{A}_{:,\tau_{j_k},:})^T*\mathcal{R}^{(k)}}
{\|\mathcal{A}_{:, \tau_{j_k},:}\|_F^2},
\]
where the step size $\alpha_k$ is chosen adaptively at each iteration.

Specifically, the adaptive step size $\alpha_k$ is determined by minimizing the residual function:
\[
\begin{aligned}
	\mathcal{G}(\alpha)
	&=
	\left\|\mathcal{R}^{(k)}-
	\alpha
	\mathcal{A}_{:,\tau_{j_k},:}*
	\frac{(\mathcal{A}_{:,\tau_{j_k},:})^T*\mathcal{R}^{(k)}}
	{\|\mathcal{A}_{:,\tau_{j_k},:}\|_F^2}
	\right\|_F^2 .
\end{aligned}
\]
By taking the derivative of $\mathcal{G}(\alpha)$ with respect to $\alpha$, it yields
\[
\frac{\mathrm{d}\mathcal{G}}{\mathrm{d}\alpha}
=
2\alpha
\left\|
\mathcal{A}_{:,\tau_{j_k},:}*
\frac{(\mathcal{A}_{:,\tau_{j_k},:})^T*\mathcal{R}^{(k)}}
{\|\mathcal{A}_{:,\tau_{j_k},:}\|_F^2}
\right\|_F^2
-
2
\left\langle
\mathcal{R}^{(k)},
\mathcal{A}_{:,\tau_{j_k},:}*
\frac{(\mathcal{A}_{:,\tau_{j_k},:})^T*\mathcal{R}^{(k)}}
{\|\mathcal{A}_{:,\tau_{j_k},:}\|_F^2}
\right\rangle .
\]
Since $\mathcal{G}(\alpha)$ is a strictly convex quadratic function of $\alpha$, the unique minimizer is given by
\begin{equation}\label{eqn:n-alpha}
	\alpha_k=
	\frac{
		\|\mathcal{A}_{:, \tau_{j_k},:}\|_F^2
		\|
		(\mathcal{A}_{:, \tau_{j_k},:})^T * \mathcal{R}^{(k)}
		\|_F^2
	}{
		\|
		\mathcal{A}_{:, \tau_{j_k},:}*
		(\mathcal{A}_{:, \tau_{j_k},:})^T * \mathcal{R}^{(k)}
		\|_F^2
	}.
\end{equation}
And the resulting tensor randomized average block coordinate descent method is summarized in Algorithm~\ref{alg:tRABCDa}.

\begin{algorithm}[!htbp]
	\caption{The tensor randomized average block coordinate descent (tRABCD) method}
	\begin{algorithmic}[1]
		\REQUIRE $\mathcal{X}^{(0)} \in \mathbb{R}^{n_2 \times p \times n_3},\,
		\mathcal{A} \in \mathbb{R}^{n_1 \times n_2 \times n_3},\,
		\mathcal{B} \in \mathbb{R}^{n_1 \times p \times n_3}$,
		a block partition of $[n_2]$: $\mathcal{T}=\{\tau_j\}_{j=1}^q$.
		\ENSURE Iterate $\mathcal{X}^{(k+1)}$.
		\STATE $\mathcal{R}^{(0)} = \mathcal{B}-\mathcal{A}*\mathcal{X}^{(0)}$.
		\FOR{$k=0,1,2,\ldots$}
		\STATE Sample $\tau_{j_k}\in\mathcal{T}$ with probability
		$\|\mathcal{A}_{:,\tau_{j_k},:}\|_F^2/\|\mathcal{A}\|_F^2$.
		\STATE Update
		$
		\mathcal{X}^{(k+1)}_{\tau_{j_k},:,:}
		=
		\mathcal{X}^{(k)}_{\tau_{j_k},:,:}
		+
				\frac{
			\|
			(\mathcal{A}_{:, \tau_{j_k},:})^T * \mathcal{R}^{(k)}
			\|_F^2
		}{
			\|
			\mathcal{A}_{:, \tau_{j_k},:}*
			(\mathcal{A}_{:, \tau_{j_k},:})^T * \mathcal{R}^{(k)}
			\|_F^2
		}
(\mathcal{A}_{:,\tau_{j_k},:})^T*\mathcal{R}^{(k)}.
		$
		\STATE $\mathcal{R}^{(k+1)}=\mathcal{R}^{(k)}-\mathcal{A}_{:,\tau_{j_k},:}*(\mathcal{X}^{(k+1)}_{\tau_{j_k},:,:}-\mathcal{X}^{(k)}_{\tau_{j_k},:,:})$.
		\ENDFOR
	\end{algorithmic}
	\label{alg:tRABCDa}
\end{algorithm}

\begin{remark}
    The computational cost per iteration of the tensor randomized average block coordinate descent method is dominated by evaluating
	\[
	(\mathcal{A}_{:,\tau_{j_k},:})^T*\mathcal{R}^{(k)}
	\quad\text{and}\quad
	\mathcal{A}_{:,\tau_{j_k},:}*\bigl((\mathcal{A}_{:,\tau_{j_k},:})^T*\mathcal{R}^{(k)}\bigr).
	\]
	Unlike the tensor randomized coordinate descent method in \cite{2026KJ}, the proposed update avoids computing the Moore--Penrose pseudoinverse $\bigl(\mathcal{A}_{:,\tau_{j_k},:}\bigr)^\dagger$, and is thus cheaper per iteration. When $\operatorname{fft}(\mathcal{A},[],3)$ is precomputed and stored, the dominant work per iteration reduces to a few matrix multiplications in the Fourier domain and a single  inverse FFT. Consequently, the computational saving sare particularly significant for relatively large block sizes $|\tau_{j_k}|$ and large $n_3$.
\end{remark}

The convergence theory of the tensor randomized average block coordinate descent method is stated in Theorem~\ref{thm:TRABCDA}.

\begin{theorem}\label{thm:TRABCDA}
	Let $
	\Gamma_1
	=
	\max_{\tau_j\in\mathcal{T}}
	\frac{
		\|\operatorname{bcirc}(\mathcal{A}_{:,\tau_j,:})\|_2^2
	}{
		\|\mathcal{A}_{:,\tau_j,:}\|_F^2
	},
	$
	and assume that $\operatorname{bcirc}(\mathcal{A})$ is of full column rank.
	Then the sequence $\{\mathcal{X}^{(k)}\}_{k=0}^{\infty}$ generated by
	the tensor randomized average block coordinate descent method satisfies
	\begin{equation}\label{eqn:4.14}
		\mathbb{E}\|\mathcal{E}^{(k+1)}\|_F^2
		\le
		\left(
		1-
		\frac{
			\sigma_{\min}^2(\operatorname{bcirc}(\mathcal{A}))
		}{
			\Gamma_1\|\mathcal{A}\|_F^2
		}
		\right)^{k+1}
		\|\mathcal{E}^{(0)}\|_F^2.
	\end{equation}
\end{theorem}

\begin{proof}
	Since $\mathcal{X}_{LS}$ is the least squares solution of the tensor system
	\eqref{eqn:taxb}, then
	\[
	\mathcal{A}^T *
	\left(
	\mathcal{B}-\mathcal{A}*\mathcal{X}_{LS}
	\right)
	=
	\mathcal{O}.
	\]
	Consequently, for any block $\tau_j\in\mathcal T$, 
	\[
	(\mathcal{A}_{:,\tau_j,:})^T *
	\left(
	\mathcal{B}-\mathcal{A}*\mathcal{X}_{LS}
	\right)
	=
	\mathcal{O}.
	\]
	Since $
	\mathcal{R}^{(k)}
	=
	\mathcal{B}-\mathcal{A}*\mathcal{X}^{(k)},
	$
	it follows that
	\begin{equation}\label{eqn:block-residual-E}
		(\mathcal{A}_{:,\tau_j,:})^T * \mathcal{R}^{(k)}
		=
		(\mathcal{A}_{:,\tau_j,:})^T * \mathcal{E}^{(k)} .
	\end{equation}
	
	Based on the step 4 of Algorithm~\ref{alg:tRABCDa}, the iterate is given by
	\[
	\mathcal{X}^{(k+1)}
	=
	\mathcal{X}^{(k)}
	+
	\alpha_k
	\frac{
		(\mathcal{A}_{:,\tau_{j_k},:})^T * \mathcal{R}^{(k)}
	}{
		\|\mathcal{A}_{:,\tau_{j_k},:}\|_F^2
	},
	\]
	where the update is applied only to the selected block $\tau_{j_k}$.
	Combining this with  Eq.~\eqref{eqn:block-residual-E}, the corresponding error-residual is 
	updated as
	\begin{equation}\label{eqn:tRABCD-E-update}
		\mathcal{E}^{(k+1)}
		=
		\mathcal{E}^{(k)}
		-
		\alpha_k
		\frac{
			\mathcal{A}_{:,\tau_{j_k},:}
			*
			(\mathcal{A}_{:,\tau_{j_k},:})^T
		}{
			\|\mathcal{A}_{:,\tau_{j_k},:}\|_F^2
		}
		*
		\mathcal{E}^{(k)} .
	\end{equation}
	
Denote
\[
\mathcal{U}_k
:=
\mathcal{A}_{:,\tau_{j_k},:}
*
(\mathcal{A}_{:,\tau_{j_k},:})^T
*
\mathcal{E}^{(k)},
\]
then Eq.~\eqref{eqn:tRABCD-E-update} is optimized  as
\[
\mathcal{E}^{(k+1)}
=
\mathcal{E}^{(k)}
-
\frac{\alpha_k}
{\|\mathcal{A}_{:,\tau_{j_k},:}\|_F^2}
\mathcal{U}_k .
\]
By taking the squared Frobenius norm on both sides, it yields
\[
\begin{aligned}
	\|\mathcal{E}^{(k+1)}\|_F^2
	&=
	\|\mathcal{E}^{(k)}\|_F^2
	+
	\frac{\alpha_k^2}
	{\|\mathcal{A}_{:,\tau_{j_k},:}\|_F^4}
	\|\mathcal{U}_k\|_F^2
	-
	\frac{2\alpha_k}
	{\|\mathcal{A}_{:,\tau_{j_k},:}\|_F^2}
	\left\langle
	\mathcal{U}_k,\mathcal{E}^{(k)}
	\right\rangle .
\end{aligned}
\]
Moreover, by the adjoint property of the $t$-product,
\[
\begin{aligned}
	\left\langle
	\mathcal{U}_k,\mathcal{E}^{(k)}
	\right\rangle
	&=
	\left\langle
	\mathcal{A}_{:,\tau_{j_k},:}
	*
	(\mathcal{A}_{:,\tau_{j_k},:})^T
	*
	\mathcal{E}^{(k)},
	\mathcal{E}^{(k)}
	\right\rangle                                      \\
	&=
	\left\langle
	(\mathcal{A}_{:,\tau_{j_k},:})^T
	*
	\mathcal{E}^{(k)},
	(\mathcal{A}_{:,\tau_{j_k},:})^T
	*
	\mathcal{E}^{(k)}
	\right\rangle                                      \\
	&=
	\left\|
	(\mathcal{A}_{:,\tau_{j_k},:})^T
	*
	\mathcal{E}^{(k)}
	\right\|_F^2 .
\end{aligned}
\]
Hence,
\[
	\|\mathcal{E}^{(k+1)}\|_F^2
	=
	\|\mathcal{E}^{(k)}\|_F^2
	+
	\frac{\alpha_k^2}
	{\|\mathcal{A}_{:,\tau_{j_k},:}\|_F^4}
	\|\mathcal{U}_k\|_F^2
	-
	\frac{2\alpha_k}
	{\|\mathcal{A}_{:,\tau_{j_k},:}\|_F^2}
	\left\|
	(\mathcal{A}_{:,\tau_{j_k},:})^T
	*
	\mathcal{E}^{(k)}
	\right\|_F^2 .
\]

From the choice of $\alpha_k$ in Eq.~\eqref{eqn:n-alpha}, when
$\mathcal{U}_k\neq\mathcal{O}$, it is obvious that
\[
\alpha_k
=
\frac{
	\|\mathcal{A}_{:,\tau_{j_k},:}\|_F^2
	\left\|
	(\mathcal{A}_{:,\tau_{j_k},:})^T
	*
	\mathcal{E}^{(k)}
	\right\|_F^2
}{
	\|\mathcal{U}_k\|_F^2
}.
\]
By substituting this expression into the preceding equality, it gives
\begin{equation}\label{eqn:tRABCD-decrease-E}
	\|\mathcal{E}^{(k+1)}\|_F^2
	=
	\|\mathcal{E}^{(k)}\|_F^2
	-
	\frac{
		\left\|
		(\mathcal{A}_{:,\tau_{j_k},:})^T
		*
		\mathcal{E}^{(k)}
		\right\|_F^4
	}{
		\|\mathcal{U}_k\|_F^2
	}.
\end{equation}
If $\mathcal{U}_k=\mathcal{O}$, then
\[
\left\|
(\mathcal{A}_{:,\tau_{j_k},:})^T
*
\mathcal{E}^{(k)}
\right\|_F^2
=
\left\langle
\mathcal{U}_k,\mathcal{E}^{(k)}
\right\rangle
=
0.
\]
In this case, $\mathcal{X}^{(k+1)} = \mathcal{X}^{(k)}$ and $\mathcal{E}^{(k+1)} = \mathcal{E}^{(k)}$, so Eq.~\eqref{eqn:tRABCD-decrease-E} holds trivially with both sides equal.

From the definition of the $t$-product, it yields
\[
	\|\mathcal{U}_k\|_F
	=
	\left\|
	\mathcal{A}_{:,\tau_{j_k},:}
	*
	(\mathcal{A}_{:,\tau_{j_k},:})^T
	*
	\mathcal{E}^{(k)}
	\right\|_F             \le
	\|\operatorname{bcirc}(\mathcal{A}_{:,\tau_{j_k},:})\|_2
	\left\|
	(\mathcal{A}_{:,\tau_{j_k},:})^T
	*
	\mathcal{E}^{(k)}
	\right\|_F .
\]
By combining this bound with Eq.~\eqref{eqn:tRABCD-decrease-E}, it is obtained that
\[
\|\mathcal{E}^{(k+1)}\|_F^2
\le
\|\mathcal{E}^{(k)}\|_F^2
-
\frac{
	\left\|
	(\mathcal{A}_{:,\tau_{j_k},:})^T
	*
	\mathcal{E}^{(k)}
	\right\|_F^2
}{
	\|\operatorname{bcirc}(\mathcal{A}_{:,\tau_{j_k},:})\|_2^2
}.
\]
Equivalently,
\[
\|\mathcal{E}^{(k+1)}\|_F^2
\le
\|\mathcal{E}^{(k)}\|_F^2
-
\frac{
	\|\mathcal{A}_{:,\tau_{j_k},:}\|_F^2
}{
	\|\operatorname{bcirc}(\mathcal{A}_{:,\tau_{j_k},:})\|_2^2
}
\frac{
	\left\|
	(\mathcal{A}_{:,\tau_{j_k},:})^T
	*
	\mathcal{E}^{(k)}
	\right\|_F^2
}{
	\|\mathcal{A}_{:,\tau_{j_k},:}\|_F^2
}.
\]
Let $
\Gamma_1
=
\max_{\tau_j\in\mathcal{T}}
\frac{
	\|\operatorname{bcirc}(\mathcal{A}_{:,\tau_j,:})\|_2^2
}{
	\|\mathcal{A}_{:,\tau_j,:}\|_F^2
},
$ then
\begin{equation}\label{eqn:tRABCD-one-step-bound-E}
	\|\mathcal{E}^{(k+1)}\|_F^2
	\le
	\|\mathcal{E}^{(k)}\|_F^2
	-
	\frac{1}{\Gamma_1}
	\frac{
		\left\|
		(\mathcal{A}_{:,\tau_{j_k},:})^T
		*
		\mathcal{E}^{(k)}
		\right\|_F^2
	}{
		\|\mathcal{A}_{:,\tau_{j_k},:}\|_F^2
	}.
\end{equation}

By taking the conditional expectation with respect to the sampled block
$\tau_{j_k}$, it yields that
\begin{align}
	\mathbb{E}\!\left[
	\|\mathcal{E}^{(k+1)}\|_F^2
	\mid
	\mathcal{E}^{(k)}
	\right]
	&\le
	\|\mathcal{E}^{(k)}\|_F^2
	-
	\frac{1}{\Gamma_1}
	\sum_{j=1}^{q}
	\frac{
		\|\mathcal{A}_{:,\tau_j,:}\|_F^2
	}{
		\|\mathcal{A}\|_F^2
	}
	\frac{
		\bigl\|
		(\mathcal{A}_{:,\tau_j,:})^T
		*
		\mathcal{E}^{(k)}
		\bigr\|_F^2
	}{
		\|\mathcal{A}_{:,\tau_j,:}\|_F^2
	}
	\nonumber\\
	&=
	\|\mathcal{E}^{(k)}\|_F^2
	-
	\frac{1}
	{\Gamma_1\|\mathcal{A}\|_F^2}
	\sum_{j=1}^{q}
	\bigl\|
	(\mathcal{A}_{:,\tau_j,:})^T
	*
	\mathcal{E}^{(k)}
	\bigr\|_F^2
	\nonumber\\
	&=
	\|\mathcal{E}^{(k)}\|_F^2
	-
	\frac{1}
	{\Gamma_1\|\mathcal{A}\|_F^2}
	\bigl\|
	\mathcal{A}^T * \mathcal{E}^{(k)}
	\bigr\|_F^2 .
	\label{eqn:tRABCD-cond-exp-E}
\end{align}

Given that
\[
\mathcal{E}^{(k)}
=
\mathcal{A}*
\left(
\mathcal{X}_{LS}-\mathcal{X}^{(k)}
\right)
\in
\operatorname{range}_p(\mathcal{A}),
\]
and $\operatorname{bcirc}(\mathcal{A})$ has full column rank, it follows
\[
\bigl\|
\mathcal{A}^T * \mathcal{E}^{(k)}
\bigr\|_F
\ge
\sigma_{\min}(\operatorname{bcirc}(\mathcal{A}))
\,
\|\mathcal{E}^{(k)}\|_F .
\]
By substituting this estimate into Eq.~\eqref{eqn:tRABCD-cond-exp-E}, it gives
\[
\mathbb{E}\!\left[
\|\mathcal{E}^{(k+1)}\|_F^2
\mid
\mathcal{E}^{(k)}
\right]
\le
\left(
1-
\frac{
	\sigma_{\min}^2(\operatorname{bcirc}(\mathcal{A}))
}{
	\Gamma_1\|\mathcal{A}\|_F^2
}
\right)
\|\mathcal{E}^{(k)}\|_F^2 .
\]

By taking the total expectation on both sides and applying the above recursive
inequality for $(k+1)$ steps, it is obtained that
\[
\mathbb{E}\|\mathcal{E}^{(k+1)}\|_F^2
\le
\left(
1-
\frac{
	\sigma_{\min}^2(\operatorname{bcirc}(\mathcal{A}))
}{
	\Gamma_1\|\mathcal{A}\|_F^2
}
\right)^{k+1}
\|\mathcal{E}^{(0)}\|_F^2 .
\]
This completes the proof.
\end{proof}

\section{Tensor randomized average block coordinate descent method with heavy-ball momentum}
\label{sec:trabcd}

Heavy-ball momentum, first introduced by Polyak \cite{1964P} in 1964, accelerates gradient-based methods by incorporating a momentum term from the previous iteration. This technique has been widely applied to accelerate randomized iterative methods \cite{2020LR, 2024ZH}. Recently, a heavy-ball-type accelerated tensor Kaczmarz method was proposed in \cite{2024LL}, demonstrating a faster convergence rate than the  tensor randomized average Kaczmarz method.

Motivated by these advancements, we incorporate the heavy-ball momentum technique into our tensor randomized average block coordinate descent method (Algorithm \ref{alg:tRABCDa}). This strategy ensures that each iteration achieves optimal residual reduction within the currently available two-dimensional affine subspace.

At the $k$-th iteration of the new method, a block $\tau_{j_k}\in\mathcal{T}$ is randomly selected with the same probability distribution as Algorithm~\ref{alg:tRABCDa}. We define a tensor $\mathcal{Z}_{\tau_{j_k}}$, whose horizontal slices are all zeros except for the slice corresponding to $\tau_{j_k}$, given by,
\begin{equation}\label{eqn:MR-Zdef}
	\mathcal{Z}_{\tau_{j_k},:,:}
	=
	\bigl(\mathcal{A}_{:,\tau_{j_k},:}\bigr)^T*\mathcal{R}^{(k)}.
\end{equation}

Based on this, the heavy-ball momentum iteration formula is defined as
\begin{equation}\label{eqn:MR-Xupdate}
	\mathcal{X}^{(k+1)}
	=
	\mathcal{X}^{(k)}
	+
	\alpha_k\,\mathcal{Z}_{\tau_{j_k}}
	+
	\beta_k\bigl(\mathcal{X}^{(k)}-\mathcal{X}^{(k-1)}\bigr),
	\qquad t\ge 1,
\end{equation}
where $\alpha_k$ is the step size and $\beta_k$ is the momentum parameter.
From Eq.~\eqref{eqn:MR-Xupdate}, the residual recursion becomes
\begin{equation}\label{eqn:MR-Rupdate}
	\mathcal{R}^{(k+1)}
	=
	\mathcal{R}^{(k)}
	-
	\alpha_k\,\mathcal{A}*\mathcal{Z}_{\tau_{j_k}}
	+
	\beta_k\bigl(\mathcal{R}^{(k)}-\mathcal{R}^{(k-1)}\bigr),
	\qquad t\ge 1,
\end{equation}
Let
\begin{equation}\label{eqn:MR-directions}
	\mathcal{U}_k := \mathcal{A}*\mathcal{Z}_{\tau_{j_k}},
	\qquad
	\mathcal{D}_k := \mathcal{R}^{(k)}-\mathcal{R}^{(k-1)},
\end{equation}
and Eq.~\eqref{eqn:MR-Rupdate} is optimized as
$$
\mathcal{E}^{(k+1)}(\alpha,\beta)
=
\mathcal{E}^{(k)}
-
\alpha\,\mathcal{U}_k
+
\beta\,\mathcal{D}_k.
$$

The parameters $\alpha_k$ and $\beta_k$ are chosen to minimize the Frobenius norm of $\mathcal{E}^{(k+1)}(\alpha,\beta)$ over the current two-dimensional affine space, that is 
$$
	(\alpha_k,\beta_k)
	=
	\arg\min_{\alpha,\beta\in\mathbb{R}}
	\left\|
	\mathcal{E}^{(k)}
	-
	\alpha\mathcal{U}_k
	+
	\beta\mathcal{D}_k
	\right\|_F^2.
$$

By the optimality condition of orthogonal projection, the optimal solution satisfies
\begin{equation*}\label{eqn:MR-orthogonality}
	\left\langle \mathcal{E}^{(k+1)},\ \mathcal{U}_k\right\rangle = 0,
	\qquad
	\left\langle \mathcal{E}^{(k+1)},\ \mathcal{D}_k\right\rangle = 0.
\end{equation*}
The closed-form expressions of $\alpha_k$ and $\beta_k$ are deduced below. Defining the objective function as
\begin{equation*}\label{eqn:MR-G}
	\mathcal{G}(\alpha,\beta)
	:=
	\frac{1}{2}\left\|
	\mathcal{E}^{(k)}
	-
	\alpha\,\mathcal{U}_k
	+
	\beta\,\mathcal{D}_k
	\right\|_F^2,
\end{equation*}
By taking partial derivatives with respect to $\alpha$ and $\beta$, and setting them to zero, it gives that
\begin{equation}\label{eqn:MR-normal-eq}
	\left\{
	\begin{aligned}
		\alpha_k\|\mathcal{U}_k\|_F^2 - \beta_k\langle \mathcal{U}_k,\mathcal{D}_k\rangle
		&= \langle \mathcal{U}_k,\mathcal{E}^{(k)}\rangle,\\
		\alpha_k\langle \mathcal{U}_k,\mathcal{D}_k\rangle - \beta_k\|\mathcal{D}_k\|_F^2
		&= \langle \mathcal{D}_k,\mathcal{E}^{(k)}\rangle.
	\end{aligned}
	\right.
\end{equation}

According to the proof of Theorem~\ref{thm:TRABCDA}, it holds that
\begin{equation}\label{eqn:MR-inner-simplify}
	\langle \mathcal{U}_k,\mathcal{E}^{(k)}\rangle=
	\|\mathcal{Z}_{\tau_{j_k}}\|_F^2.
\end{equation}
Furthermore, the residual update satisfies
\[
\mathcal D_k
=
\mathcal R^{(k)}-\mathcal R^{(k-1)}
=
-\alpha_{k-1}\mathcal U_{k-1}
+
\beta_{k-1}\mathcal D_{k-1},
\]
which implies that $
\mathcal D_k
\in
\operatorname{span}\{\mathcal U_{k-1},\mathcal D_{k-1}\}.
$
Due to the first-order optimality condition of the minimal residual problem at $(k-1)$-th
iteration, $\mathcal E^{(k)}$ is orthogonal to
$\operatorname{span}\{\mathcal U_{k-1},\mathcal D_{k-1}\}$. Consequently,
\begin{equation}\label{eqn:DR-R-orthogonal}
	\langle \mathcal{D}_k,\mathcal{E}^{(k)}\rangle = 0.
\end{equation}

To simplify the notation, we define
\begin{equation}\label{eqn:MR-theta}
	\theta_k
	:=
	\|\mathcal{U}_k\|_F^2\ \|\mathcal{D}_k\|_F^2
	-
	\langle \mathcal{U}_k,\mathcal{D}_k\rangle^2.
\end{equation}
When $\theta_k=0$, it is observed that $\mathcal{U}_k$ and $\mathcal{D}_k$ are linearly dependent, so the direction subspace degenerates into a one-dimensional space. In this case, $\beta_k$ is set to zero, and a one-dimensional minimization along $\mathcal{U}_k$ is performed as follows:
\begin{equation}\label{eqn:MR-alpha-degenerate}
	\alpha_k
=
\frac{\langle \mathcal{U}_k,\mathcal{E}^{(k)}\rangle}{\|\mathcal{U}_k\|_F^2}
=
\frac{\|\mathcal{Z}_{\tau_{j_k}}\|_F^2}{\|\mathcal{A}*\mathcal{Z}_{\tau_{j_k}}\|_F^2}.
\end{equation}
When $\theta_k>0$, Eq.~\eqref{eqn:MR-normal-eq} admits a unique solution.

By using Eq.~\eqref{eqn:MR-inner-simplify} and  Eq.~\eqref{eqn:DR-R-orthogonal}, the optimal parameters are obtained as
\begin{equation}\label{eqn:MR-alpha-beta}
	\left\{
	\begin{aligned}
		\alpha_k
		&=
		\frac{
			\|\mathcal{Z}_{\tau_{j_k}}\|_F^2\ \|\mathcal{D}_k\|_F^2
		}{
			\theta_k
		},\\[0.3em]
		\beta_k
		&=
		\frac{
			\|\mathcal{Z}_{\tau_{j_k}}\|_F^2\ \langle \mathcal{U}_k,\mathcal{D}_k\rangle
		}{
			\theta_k
		}.
	\end{aligned}
	\right.
\end{equation}

Building upon the derivations above, we propose the tensor randomized average block coordinate descent method with heavy-ball momentum, detailed in Algorithm~\ref{alg:tRABCD-HB}.

\begin{algorithm}[H]
	\caption{The tensor randomized average block coordinate descent method with heavy-ball momentum (tRABCD-HB)}
	\label{alg:tRABCD-HB}
	\begin{algorithmic}[1]
		\REQUIRE Initial guess $\mathcal{X}^{(0)}\in\mathbb{R}^{n_2\times p\times n_3}$, tensors $\mathcal{A}\in\mathbb{R}^{n_1\times n_2\times n_3}$, $\mathcal{B}\in\mathbb{R}^{n_1\times p\times n_3}$, block partition $\mathcal{T}=\{\tau_j\}_{j=1}^q$.
		\ENSURE Iterate $\mathcal{X}^{(k+1)}$.
		\STATE  $\mathcal{R}^{(0)}=\mathcal{B}-\mathcal{A}*\mathcal{X}^{(0)}$.
		\STATE $\mathcal{X}^{(-1)}=\mathcal{X}^{(0)}$, $\mathcal{R}^{(-1)}=\mathcal{R}^{(0)}$ (so that the momentum term is zero at $k=0$).
		\FOR{$k=0,1,2,\ldots$}
		\STATE Select an index $j_k\in\{1,\ldots,q\}$ with probability $\mathbb{P}(j_k=j)=\|\mathcal{A}_{:,\tau_j,:}\|_F^2/\|\mathcal{A}\|_F^2$, and set the selected block as $\tau_{j_k}$,
		\STATE Construct $\mathcal{Z}_{\tau_{j_k}}$: all horizontal slices except $\tau_{j_k}$ are set to zero, and $\mathcal{Z}_{\tau_{j_k},:,:}=\bigl(\mathcal{A}_{:,\tau_{j_k},:}\bigr)^T*\mathcal{R}^{(k)}$,
		\STATE Compute $\mathcal{U}_k=\mathcal{A}*\mathcal{Z}_{\tau_{j_k}}$ and $\mathcal{D}_k=\mathcal{R}^{(k)}-\mathcal{R}^{(k-1)}$,
		\STATE Compute $\theta_k=\|\mathcal{U}_k\|_F^2\|\mathcal{D}_k\|_F^2-\langle \mathcal{U}_k,\mathcal{D}_k\rangle^2$,
		\IF{$\theta_k=0$}
		\STATE $\beta_k=0$ and $\alpha_k=\|\mathcal{Z}_{\tau_{j_k}}\|_F^2/\|\mathcal{U}_k\|_F^2$.
		\ELSE
		\STATE Compute $\alpha_k,\beta_k$ by Eq.~\eqref{eqn:MR-alpha-beta}.
		\ENDIF
		\STATE $\mathcal{X}^{(k+1)}
		=
		\mathcal{X}^{(k)}
		+
		\alpha_k\,\mathcal{Z}_{\tau_{j_k}}
		+
		\beta_k\bigl(\mathcal{X}^{(k)}-\mathcal{X}^{(k-1)}\bigr)$,
		\STATE $\mathcal{R}^{(k+1)}
		=
		\mathcal{R}^{(k)}
		-
		\alpha_k\,\mathcal{U}_k
		+
		\beta_k\,\mathcal{D}_k$.
		\ENDFOR
	\end{algorithmic}
\end{algorithm}

\begin{remark}
	Note that only Frobenius norms and inner products are involved in Eq.~\eqref{eqn:MR-alpha-beta}, which are readily evaluated once $\mathcal{U}_k$ and $\mathcal{D}_k$ are computed. When $\theta_k$ is small, $\mathcal{U}_k$ and $\mathcal{D}_k$ become nearly linearly dependent, and the degenerate form \eqref{eqn:MR-alpha-degenerate} is adopted to avoid numerical instability.
\end{remark}

For the selected block $\tau_{j_k}$, we denote its cardinality by $s_k:=|\tau_{j_k}|.$ To ensure computational efficiency, the t-product is evaluated in the Fourier domain. Since
$\mathcal{A}$ remains fixed throughout the iteration, its Fourier transform
\[
\widehat{\mathcal{A}}=\operatorname{fft}(\mathcal{A},[],3)
\]
is precomputed before the iteration. Moreover, the transformed iterates  $\widehat{\mathcal{R}}^{(k)}$ and $\widehat{\mathcal{X}}^{(k)}$ are maintained and updated in the Fourier domain, so that no additional FFT or inverse FFT is needed at each iteration. Under this implementation, each t-product reduces to
matrix multiplications between frontal slices in the Fourier domain.

\begin{table}[!htbp]
	\centering
	\small
	\renewcommand{\arraystretch}{1.8} 
	\caption{Per-iteration flops of Algorithm~\ref{alg:tRABCDa} and Algorithm~\ref{alg:tRABCD-HB}.}
	\label{tab:fourier-complexity-comparison}
	\begin{tabular}{>{\centering\arraybackslash}p{0.35\textwidth} >{\centering\arraybackslash}p{0.35\textwidth}}
		\toprule
		Method & Per-iteration flops \\
		\midrule
		Algorithm~\ref{alg:tRABCDa}
		&
		$O\bigl( n_1s_kpn_3 \bigr)$
		\\
		Algorithm~\ref{alg:tRABCD-HB}
		&
		$O\bigl( n_1s_kpn_3+n_1pn_3+n_2pn_3 \bigr)$
		\\
		Additional cost of Algorithm~\ref{alg:tRABCD-HB}
		&
		$O(n_1pn_3+n_2pn_3)$
		\\
		\bottomrule
	\end{tabular}
\end{table}

The per-iteration computational complexity is summarized in
Table~\ref{tab:fourier-complexity-comparison}. Under the Fourier-domain implementation with precomputed
$\widehat{\mathcal{A}}$ and maintained Fourier-domain iterates, the additional
per-iteration cost of Algorithm~\ref{alg:tRABCD-HB} over
Algorithm~\ref{alg:tRABCDa} is
\[
O(n_1pn_3+n_2pn_3).
\]
This additional cost comes from the computation of
$\mathcal{D}_k$, the inner products and norms required by
the two-dimensional minimal residual projection and the full heavy-ball
momentum update of $\widehat{\mathcal{X}}^{(k)}$.

\subsection{Convergence Analysis}

The expected convergence bound of the proposed tensor  randomized average block coordinate descent method with heavy-ball momentum is established in this section.
A useful equality is firstly introduced below.
\begin{lemma}\label{lem:V-inner-identity}
	Let $\mathcal{U}_k:=\mathcal{A}*\mathcal{Z}_{\tau_{j_k}},
	\mathcal{D}_k:=\mathcal{R}^{(k)}-\mathcal{R}^{(k-1)},
	$
	and 
    $
	\mathcal{V}_k
	:=
	\langle \mathcal{U}_k,\mathcal{D}_k\rangle \mathcal{U}_k
	-
	\|\mathcal{U}_k\|_F^2 \mathcal{D}_k .
	$
	Let
	$
	\widetilde{\mathcal{E}}^{(k+1)}
	=
	\mathcal{E}^{(k)}
	-
	\frac{\|\mathcal{Z}_{\tau_{j_k}}\|_F^2}
	{\|\mathcal{U}_k\|_F^2}
	\mathcal{U}_k
	$
	be the contracted residual generated by Algorithm~\ref{alg:tRABCDa}.
	Then
	\begin{equation}\label{eqn:V-inner-identity-lemma}
		\left\langle
		\widetilde{\mathcal{E}}^{(k+1)},
		\mathcal{V}_k
		\right\rangle
		=
		\|\mathcal{Z}_{\tau_{j_k}}\|_F^2
		\langle
		\mathcal{U}_k,
		\mathcal{D}_k
		\rangle .
	\end{equation}
\end{lemma}

\begin{proof}
	By definition,
	\[
	\left\langle \widetilde{\mathcal{E}}^{(k+1)}, \mathcal{V}_k \right\rangle
	=
	\left\langle \mathcal{E}^{(k)} - \frac{\|\mathcal{Z}_{\tau_{j_k}}\|_F^2}{\|\mathcal{U}_k\|_F^2} \mathcal{U}_k,\, \mathcal{V}_k \right\rangle
	=
	\left\langle \mathcal{E}^{(k)}, \mathcal{V}_k \right\rangle
	-
	\frac{\|\mathcal{Z}_{\tau_{j_k}}\|_F^2}{\|\mathcal{U}_k\|_F^2} \left\langle \mathcal{U}_k, \mathcal{V}_k \right\rangle.
	\]
	Furthermore, it is easy to verify from the definition of $\mathcal{V}_k$ that it is orthogonal to $\mathcal{U}_k$:
	\[
	\left\langle \mathcal{U}_k, \mathcal{V}_k \right\rangle
	= \langle \mathcal{U}_k, \langle \mathcal{U}_k, \mathcal{D}_k \rangle \mathcal{U}_k - \|\mathcal{U}_k\|_F^2 \mathcal{D}_k \rangle
	= \langle \mathcal{U}_k, \mathcal{D}_k \rangle \|\mathcal{U}_k\|_F^2 - \|\mathcal{U}_k\|_F^2 \langle \mathcal{U}_k, \mathcal{D}_k \rangle = 0.
	\]
	Hence,
	\[
	\left\langle \widetilde{\mathcal{E}}^{(k+1)}, \mathcal{V}_k \right\rangle = \left\langle \mathcal{E}^{(k)}, \mathcal{V}_k \right\rangle.
	\]
	
	By using $\langle \mathcal{E}^{(k)}, \mathcal{D}_k \rangle = 0$ and $\langle \mathcal{E}^{(k)}, \mathcal{U}_k \rangle = \|\mathcal{Z}_{\tau_{j_k}}\|_F^2$, it holds that
	\[
	\left\langle \widetilde{\mathcal{E}}^{(k+1)}, \mathcal{V}_k \right\rangle
	=
	\|\mathcal{Z}_{\tau_{j_k}}\|_F^2 \, \langle \mathcal{U}_k, \mathcal{D}_k \rangle.
	\]
\end{proof}

\begin{theorem}\label{thm:tRABCD-HB}
	Assume that $\operatorname{bcirc}(\mathcal{A})$ is full column rank and
	$\theta_k$ is defined by Eq.~\eqref{eqn:MR-theta}. The sequence $\{\mathcal{X}^{(k)}\}_{k=0}^{\infty}$ generated by
	Algorithm~\ref{alg:tRABCD-HB} satisfies
	\begin{equation}\label{eqn:tRABCD-HB}
		\mathbb{E}\|\mathcal{E}^{(k+1)}\|_F^2
		\le
		\left[
		(1-\Gamma_2^2)
		\left(
		1-\frac{\sigma_{\min}^2(\operatorname{bcirc}(\mathcal{A}))}
		{\Gamma_1\|\mathcal{A}\|_F^2}
		\right)
		\right]^{k+1}
		\|\mathcal{E}^{(0)}\|_F^2,
	\end{equation}
	where $
		\Gamma_2^2
		:=
		\inf_k
		\frac{
			\left\langle
			\widetilde{\mathcal{E}}^{(k+1)},
			\mathcal{V}_k
			\right\rangle^2
		}{
			\|\mathcal{V}_k\|_F^2
			\|\widetilde{\mathcal{E}}^{(k+1)}\|_F^2
		}
	$,  $\widetilde{\mathcal{E}}^{(k+1)}$ and $\mathcal{V}_k$ are defined in Lemma~\ref{lem:V-inner-identity}.
\end{theorem}

\begin{proof}
	If the case $\theta_k=0$ occurs frequently, the iteration reduces to a one-dimensional minimal residual update, and its convergence is guaranteed by the analysis of Algorithm~\ref{alg:tRABCDa}. Therefore, only the general case $\theta_k > 0$ is considered below.
	
	Let $\tilde{\mathcal{X}}^{(k+1)}$ denote the iterate obtained by performing one step of Algorithm~\ref{alg:tRABCDa} from $\mathcal{X}^{(k)}$. By Theorem~\ref{thm:TRABCDA}, the following bound satisfies:
	\begin{equation}\label{eqn:n-4}
		\mathbb{E}\|\widetilde{\mathcal{E}}^{(k+1)}\|_F^2
		\le
		\left(1-\frac{\sigma^2_{\min}(\operatorname{bcirc}(\mathcal{A}))}{\Gamma_1\|\mathcal{A}\|_F^2}\right)
		\|\mathcal{E}^{(k)}\|_F^2.
	\end{equation}
	
	Since both $\widetilde{\mathcal{E}}^{(k+1)}$ and $\mathcal{E}^{(k+1)}$ belong to the affine space
	\[
	\mathcal{E}^{(k)}+\operatorname{span}\{\mathcal{U}_k,\mathcal{D}_k\},
	\]
	and the optimal projection condition ensures that
	\[
	\mathcal{E}^{(k+1)}\perp \operatorname{span}\{\mathcal{U}_k,\mathcal{D}_k\},
	\]
	it follows that
	\[
	\widetilde{\mathcal{E}}^{(k+1)}-\mathcal{E}^{(k+1)}
	\in \operatorname{span}\{\mathcal{U}_k,\mathcal{D}_k\},\quad
	\left\langle
	\mathcal{E}^{(k+1)},
	\widetilde{\mathcal{E}}^{(k+1)}-\mathcal{E}^{(k+1)}
	\right\rangle=0.
	\]
	Hence, by the Pythagorean theorem,
	\begin{equation}\label{eqn:n-norme}
		\|\mathcal{E}^{(k+1)}\|_F^2
		=
		\|\widetilde{\mathcal{E}}^{(k+1)}\|_F^2
		-
		\|\widetilde{\mathcal{E}}^{(k+1)}-\mathcal{E}^{(k+1)}\|_F^2.
	\end{equation}
	
	From the iteration formula, it holds that
	\[
	\widetilde{\mathcal{E}}^{(k+1)}  = \mathcal{E}^{(k)}-\frac{\|(\mathcal{A}_{:, \tau_{j_k},:})^T * \mathcal{R}^{(k)}\|_F^2}{\|\mathcal{A}_{:, \tau_{j_k},:}*(\mathcal{A}_{:, \tau_{j_k},:})^T * \mathcal{R}^{(k)}\|_F^2}\mathcal{A}_{:, \tau_{j_k},:}*(\mathcal{A}_{:, \tau_{j_k},:})^T * \mathcal{R}^{(k)}.
	\]
	
	According to Eq.~\eqref{eqn:MR-theta}, it gives
	\begin{align*}
		\mathcal{E}^{(k+1)}
		&= \mathcal{E}^{(k)}
		- \frac{\|\mathcal{Z}_{\tau_{j_k}}\|_F^2 \|\mathcal{D}_k\|_F^2}{\theta_k}  \mathcal{U}_k  + \frac{\|\mathcal{Z}_{\tau_{j_k}}\|_F^2
			\langle \mathcal{U}_k , \mathcal{D}_k \rangle}{\theta_k}
		\mathcal{D}_k \notag \\
		&= \widetilde{\mathcal{E}}^{(k+1)}
		- \frac{\|\mathcal{Z}_{\tau_{j_k}}\|_F^2
			\langle\mathcal{U}_k , \mathcal{D}_k \rangle^2}
		{\|\mathcal{A} * \mathcal{Z}_{\tau_{j_k}}\|_F^2 \theta_k}
		\mathcal{U}_k  + \frac{\|\mathcal{Z}_{\tau_{j_k}}\|_F^2
			\langle\mathcal{U}_k , \mathcal{D}_k \rangle}
		{\|\mathcal{A} * \mathcal{Z}_{\tau_{j_k}}\|_F^2 \theta_k}
		\|\mathcal{U}_k \|_F^2
		\mathcal{D}_k.
	\end{align*}
	
	Let
	\[
	\mathcal{V}_k := \langle \mathcal{U}_k, \mathcal{D}_k \rangle \mathcal{U}_k - \|\mathcal{U}_k\|_F^2 \mathcal{D}_k,
	\]
	then
	\begin{align}
		\mathcal{E}^{(k+1)}
		= \widetilde{\mathcal{E}}^{(k+1)}
		- \frac{\|\mathcal{Z}_{\tau_{j_k}}\|_F^2
			\langle\mathcal{U}_k, \mathcal{D}_k \rangle}
		{\|\mathcal{A} * \mathcal{Z}_{\tau_{j_k}}\|_F^2 \theta_k}
		\mathcal{V}_k. \label{eqn:3.18}
	\end{align}
	
	By substituting $\| \mathcal{V}_k \|_F^2 = \| \mathcal{A} \ast \mathcal{Z}_{\tau_{j_k}} \|_F^2 \theta_k$ and Eq.~\eqref{eqn:V-inner-identity-lemma} into Eq.~\eqref{eqn:3.18}, it gives
	\begin{align*}
		\mathcal{E}^{(k+1)}
		= \widetilde{\mathcal{E}}^{(k+1)}
		- \frac{\langle \widetilde{\mathcal{E}}^{(k+1)}, \mathcal{V}_k \rangle}
		{\|\mathcal{V}_k\|_F^2}
		\mathcal{V}_k.
	\end{align*}
	Therefore,
	\begin{align}
		\|\mathcal{E}^{(k+1)} - \widetilde{\mathcal{E}}^{(k+1)}\|_F^2
		&= \frac{\langle \widetilde{\mathcal{E}}^{(k+1)}, \mathcal{V}_k \rangle^2}
		{\|\mathcal{V}_k\|_F^2} \notag \\
		&= \frac{\langle \widetilde{\mathcal{E}}^{(k+1)}, \mathcal{V}_k \rangle^2}
		{\|\mathcal{V}_k\|_F^2 \|\widetilde{\mathcal{E}}^{(k+1)}\|_F^2}
		\|\widetilde{\mathcal{E}}^{(k+1)}\|_F^2.
		\label{eqn:328}
	\end{align}
	
	Based on Eq.~\eqref{eqn:328}, it is obtained that
	\begin{align}
		\|\mathcal{E}^{(k+1)}\|_F^2
		\le (1 - \Gamma_2^2)
		\|\widetilde{\mathcal{E}}^{(k+1)}\|_F^2,
	\end{align}
	where $
	\Gamma_2^2
	:=
	\inf_k
	\frac{
		\left\langle
		\widetilde{\mathcal{E}}^{(k+1)},
		\mathcal{V}_k
		\right\rangle^2
	}{
		\|\mathcal{V}_k\|_F^2
		\|\widetilde{\mathcal{E}}^{(k+1)}\|_F^2
	}
	$.
	
	By taking conditional expectation, it follows that
	\begin{align*}
		\mathbb{E}\!\left[ \|\mathcal{E}^{(k+1)}\|_F^2 \mid \mathcal{X}^{(k)} \right]
		\le (1 - \Gamma_2^2)\,
		\mathbb{E}\!\left[ \|\widetilde{\mathcal{E}}^{(k+1)}\|_F^2 \mid \mathcal{X}^{(k)} \right].
	\end{align*}
	Then by taking full expectation and combining with Eq.~\eqref{eqn:n-4}, it is obvious that
	\begin{align}
		\mathbb{E}\|\mathcal{E}^{(k+1)}\|_F^2
		\le (1 - \Gamma_2^2)
		\left( 1 - \frac{\sigma_{\min}^2(\operatorname{bcirc}(\mathcal{A}))}
		{\Gamma_1 \|\mathcal{A}\|_F^2} \right)
		\|\mathcal{E}^{(k)}\|_F^2.
		\label{eqn:330}
	\end{align}
	
	Finally, the proof is completed by applying the above inequality recursively for $(k+1)$ steps.
\end{proof}

\begin{remark}
	By Cauchy--Schwarz inequality, for every index $k$ such that
	$\mathcal{V}_k\ne 0$ and $\widetilde{\mathcal{E}}^{(k+1)}\ne 0$, it gives that
	\[
	0\le
	\frac{
		\left\langle
		\widetilde{\mathcal{E}}^{(k+1)},\mathcal{V}_k
		\right\rangle^2
	}{
		\|\mathcal{V}_k\|_F^2
		\|\widetilde{\mathcal{E}}^{(k+1)}\|_F^2
	}
	\le 1,
	\]
	where equality holds if and only if
	$\widetilde{\mathcal{E}}^{(k+1)}$ and $\mathcal{V}_k$ are linearly dependent.
	Hence $\Gamma_2^2=1$ occur only in the degenerate case where
	\[
	\widetilde{\mathcal{E}}^{(k+1)}
	\in \operatorname{span}\{\mathcal{V}_k\}
	\quad \text{for all admissible } k .
	\]
	In this case, the
	corresponding error after the projection step becomes zero,  and the method reaches the exact solution in finite steps. Furthermore, in the nondegenerate case where
	there exists at least one admissible index $k$ such that
	$\widetilde{\mathcal{E}}^{(k+1)}$ and $\mathcal{V}_k$ are not linearly dependent,
	one has
	\[
	0\le \Gamma_2^2<1.
	\]
	
	Therefore, Theorem~\ref{thm:tRABCD-HB} is understood as a nontrivial linear
	convergence result when $\Gamma_2^2<1$, while the limiting case
	$\Gamma_2^2=1$ is still consistent with the theorem.
\end{remark}

\begin{table}[!htbp]
    \centering
    \small 
    \setlength{\tabcolsep}{4pt} 
    \renewcommand{\arraystretch}{1.8} 
    \caption{Summary of convergence rates.}
    \label{tab:convergence_summary}
    \begin{tabular}{c c c}
        \toprule
        Method & Convergence Rate Upper Bound & Key Condition / Parameter \\
        \midrule 
        tRABCD (Algorithm~\ref{alg:tRABCDa})
        & $1 - \dfrac{\sigma_{\min}^2(\operatorname{bcirc}(\mathcal{A}))}{\Gamma_1 \|\mathcal{A}\|_F^2}$
        & $\Gamma_1 = \max\limits_{\tau_j \in \mathcal{T}} \dfrac{\|\operatorname{bcirc}(\mathcal{A}_{:,\tau_j,:})\|_2^2}{\|\mathcal{A}_{:,\tau_j,:}\|_F^2}$ \\
        tRABCD-HB (Algorithm~\ref{alg:tRABCD-HB})
        & $(1 - \Gamma_2^2)\left(1 - \dfrac{\sigma_{\min}^2(\operatorname{bcirc}(\mathcal{A}))}{\Gamma_1 \|\mathcal{A}\|_F^2}\right)$
        & $\Gamma_2 = \min\limits_k \dfrac{\langle \widetilde{\mathcal{E}}^{(k+1)}, \mathcal{V}_k \rangle}{\|\mathcal{V}_k\|_F \|\widetilde{\mathcal{E}}^{(k+1)}\|_F}$ \\
        \bottomrule
    \end{tabular}
\end{table}

The upper bounds of the convergence rates are summarized in Table~\ref{tab:convergence_summary}, along with the associated key constants of two tensor randomized average block coordinate descent methods. It is observed from Table~\ref{tab:convergence_summary} that the convergence factor of the tensor randomized average block coordinate descent method with heavy-ball momentum is multiplied by an additional factor $(1-\Gamma_2^2)$, and thus it admits a smaller upper bound.

Under the angle induced by the Frobenius inner product, it holds that
\[
\gamma_k :=
\frac{
	\left\langle
	\widetilde{\mathcal{E}}^{(k+1)},
	\mathcal{V}_k
	\right\rangle
}
{
	\|\widetilde{\mathcal{E}}^{(k+1)}\|_F
	\,\|\mathcal{V}_k\|_F
}
=
\cos\!\Bigl(
\angle\bigl(
\widetilde{\mathcal{E}}^{(k+1)},
\mathcal{V}_k
\bigr)
\Bigr)
\in [-1,1].
\]
Therefore, the factor $(1-\Gamma_2^2)$ admits a natural interpretation as worst-case $\sin^2$ term. When the two-dimensional subspace is capable of capturing more effective descent directions, $\Gamma_2$ becomes larger, and the advantage of the tensor randomized average block coordinate descent method with heavy-ball momentum over the tensor randomized average block coordinate descent method becomes more significant.

\section{Numerical Experiments}
\label{sec:num}
In this section, numerical experiments are conducted to verify the effectiveness of the proposed method. The number of iterations (denoted by ``IT'') and the computational time (denoted by ``CPU'') are adopted as evaluation metrics. The tensor least squares problem is constructed as
\[
\mathcal{B} = \mathcal{A} * \mathcal{X} + \mathcal{N}, \quad
\mathcal{A} \in \mathbb{R}^{n_1 \times n_2 \times n_3}, \quad
\mathcal{X} \in \mathbb{R}^{n_2 \times p \times n_3}, \quad
\mathcal{B} \in \mathbb{R}^{n_1 \times p \times n_3}.
\]
Experiments are performed on tensors $\mathcal{A}$ of varying sizes and tensors $\mathcal{X}$
drawn from different sources, including random tensors, video data, and hyperspectral images.
Unless stated otherwise, the entries of $\mathcal{A}$, $\mathcal{X}$ and the noise tensor $\mathcal{N}$ are sampled from the standard normal distribution, where the noise tensor $\mathcal{N}$ satisfies $\|\mathcal{N}\|_F = 10^{-2} \times \|\mathcal{A} * \mathcal{X}\|_F.$

All methods are initialized with $\mathcal{X}^{(0)} = \mathcal{O}$. The iteration is terminated when the number of iterations exceeds $10000$ or when the relative solution error (RSE) satisfies
\[
\operatorname{RSE} = \frac{\|\mathcal{X}^{(k)} - \mathcal{X}_{LS}\|_F^2}{\|\mathcal{X}_{LS}\|_F^2}
= \frac{\|\widehat{\mathcal{X}}^{(k)} - \widehat{\mathcal{X}}_{LS}\|_F^2}{\|\widehat{\mathcal{X}}_{LS}\|_F^2}
\le 10^{-6}.
\]

First, the tensor randomized average block coordinate descent (tRABCD) method is compared with the tensor randomized block extended Kaczmarz (tRBEK) method \cite{2024HZ}.
For the tensor randomized average block coordinate descent method, the lateral slice partition $\mathcal{T} = \{\tau_i\}_{i=1}^q$ of $\mathcal{A}$ is defined by
\[
\begin{aligned}
\tau_i &= \{(i-1)s+1, (i-1)s+2, \ldots, is\}, \quad i = 1, 2, \ldots, q-1, \\
\tau_q &= \{(q-1)s+1, (q-1)s+2, \ldots, n_2\}.
\end{aligned}
\]
The same lateral partition is adopted for the tensor randomized block extended Kaczmarz method, and its horizontal slice partition is defined as
\[
\begin{aligned}
I_i &= \{(i-1)s+1, (i-1)s+2, \ldots, is\}, \quad i = 1, 2, \ldots, q-1, \\
I_q &= \{(q-1)s+1, (q-1)s+2, \ldots, n_1\}.
\end{aligned}
\]

For a fair comparison, an  inverse-free format of the tensor randomized block extended Kaczmarz method is  employed, in which the step size parameters are adaptively chosen to minimize the approximation error at each iteration. The resulting  iterative formula is shown as follows:
\begin{equation}
\left\{
\begin{aligned}
\mathcal{Y}^{(k)} &= \mathcal{Y}^{(k-1)} - \alpha_{k1} \left( \mathcal{A}_{:,J_j,:} * \mathcal{A}_{:,J_j,:}^T * \mathcal{Y}^{(k-1)} \right), \\
\mathcal{X}^{(k)} &= \mathcal{X}^{(k-1)} + \alpha_{k2} \left( \mathcal{A}_{I_i,:,:}^T * \mathcal{H}^{(k)} \right),
\end{aligned}
\right.
\end{equation}
where
$\alpha_{k1} = \dfrac{\|\mathcal{A}_{:,J_j,:}^T * \mathcal{Y}^{(k-1)}\|_F^2}{\|\mathcal{A}_{:,J_j,:} * \mathcal{A}_{:,J_j,:}^T * \mathcal{Y}^{(k-1)}\|_F^2}$,
$\mathcal{H}^{(k)} = \mathcal{B}_{I_i,:,:} - \mathcal{Y}^{(k)}_{I_i,:,:} - \mathcal{A}_{I_i,:,:} * \mathcal{X}^{(k-1)}$,
and
$\alpha_{k2} = \dfrac{\|\mathcal{H}^{(k)}\|_F^2}{\|\mathcal{A}_{I_i,:,:}^T * \mathcal{H}^{(k)}\|_F^2}$.

\begin{example}\label{ex:4.1}
The optimal block size $s$ for the tensor randomized block extended Kaczmarz method and the tensor randomized average block coordinate descent method is investigated under the setting $n_1 = 500$, $n_2 = 100$, $n_3 = 10$ and $p = 30$.
\end{example}

\begin{figure}[!htbp]
    \centering
    \includegraphics[width=0.48\textwidth]{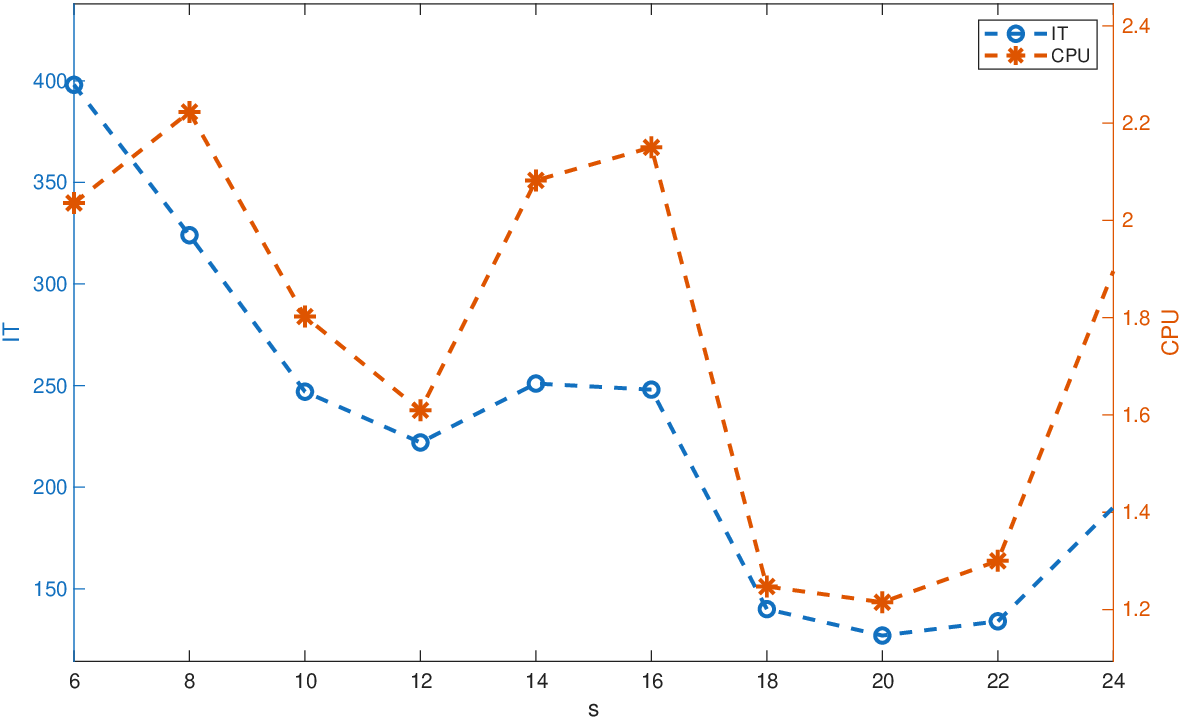}
    \hfill
    \includegraphics[width=0.48\textwidth]{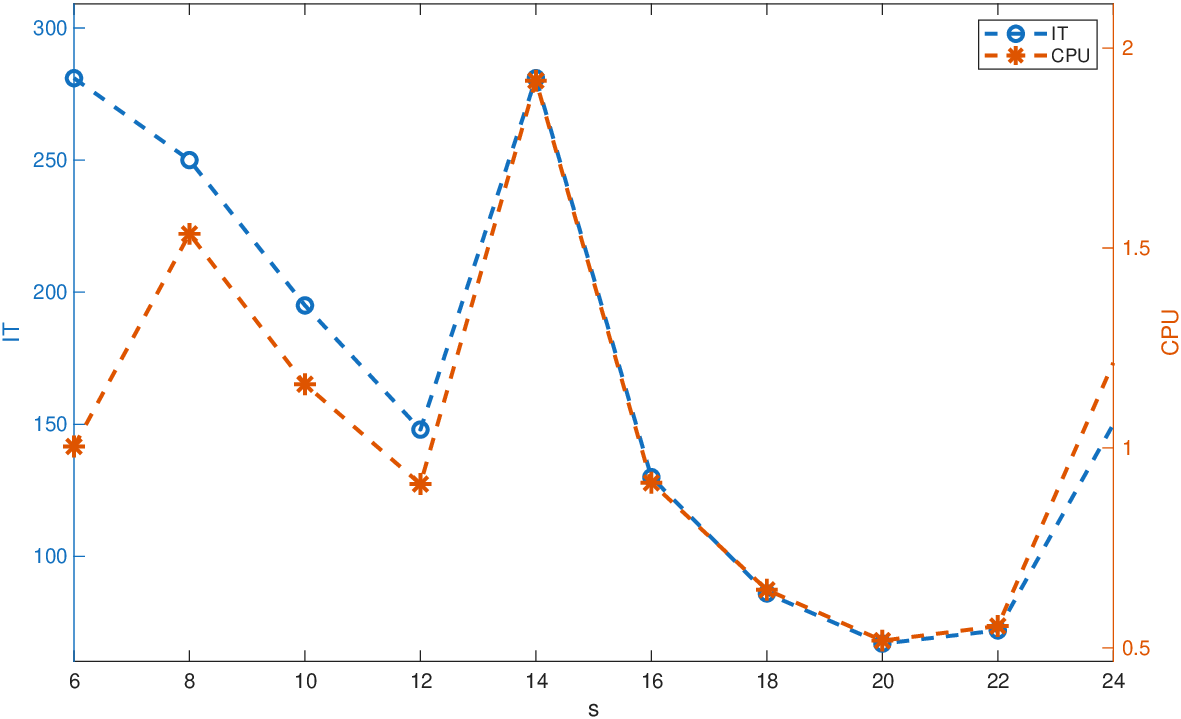}
    \caption{The curves of iteration counts and computational time of the tested methods versus the block size (left: tRBEK, right: tRABCD).}
    \label{fig:4.1}
\end{figure}

It is observed from Fig.~\ref{fig:4.1} that when $s = 20 = 0.2 n_2$, the tensor randomized average block coordinate descent method and the tensor randomized block extended Kaczmarz method achieve the fewest iterations and shortest CPU time.  Consequently, the block size is fixed as $s = 0.2n_2$ for all block-based methods in the subsequent experiments.

Next, the effectiveness of the tensor randomized average block coordinate descent method with heavy-ball momentum (tRABCD-HB) is evaluated, with the tensor randomized average block coordinate descent method serving as the baseline method. Both methods employ the block partition strategy described above. To quantify the acceleration effect introduced by the heavy-ball momentum, the speed-up ratio is defined as
\[
\operatorname{speed\text{-}up} = \frac{\operatorname{CPU}_{\text{tRABCD}}}{\operatorname{CPU}_{\text{tRABCD-HB}}}.
\]
\begin{example}\label{ex:6.0}
  The tensor randomized average block coordinate descent method with adaptive heavy-ball momentum, as defined in Eq.~\eqref{eqn:MR-alpha-beta}, is compared against its variants using fixed momentum parameters ($\beta = 0.1:0.05:0.35$) in this experiment. The parameters are set to $n_1 = 100$, $n_2 = 20$, $p=10$ and $n_3 = 10$.
\end{example}

\begin{figure}[!htbp]
    \centering
    \includegraphics[width=0.48\textwidth]{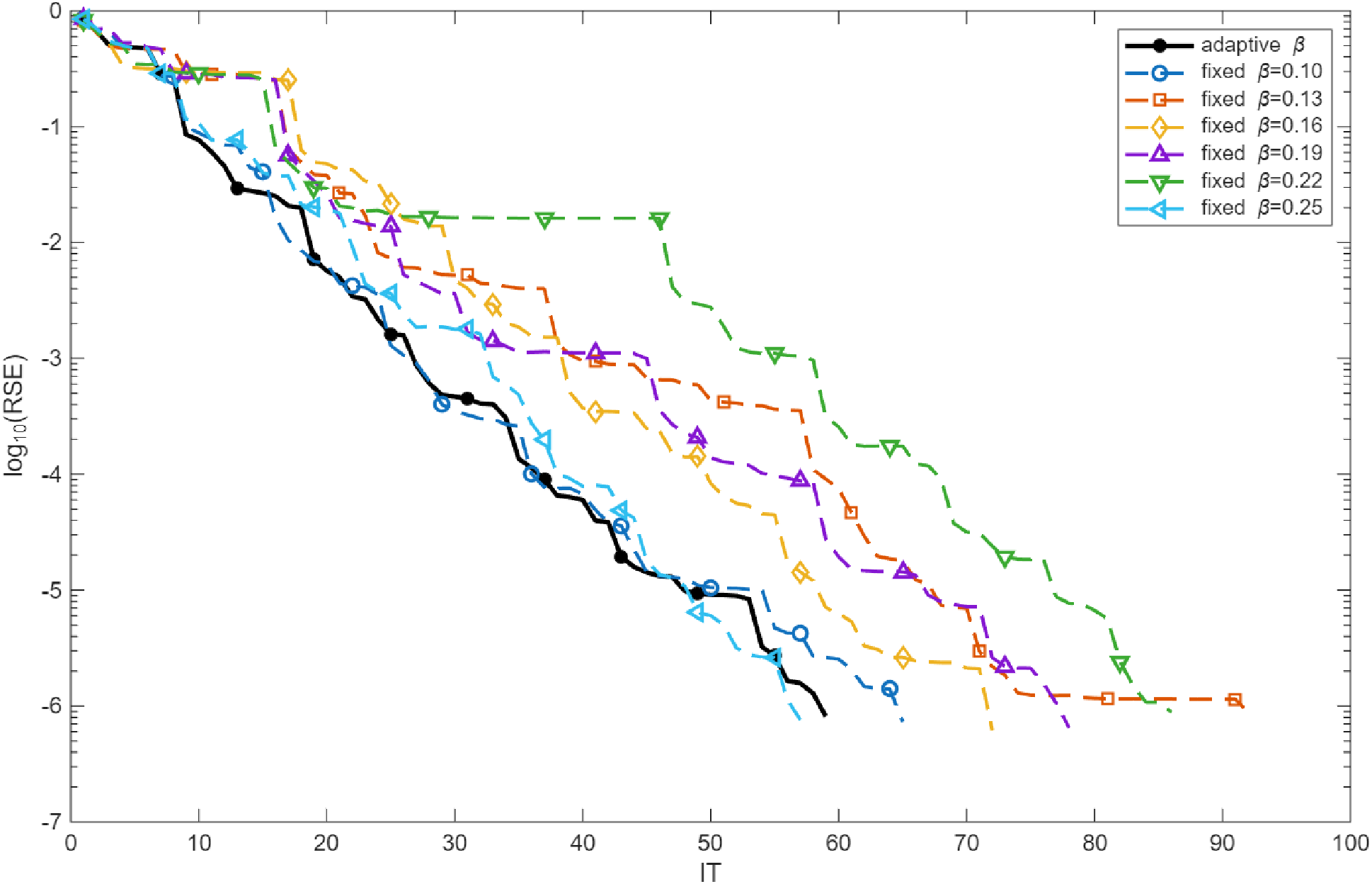}
    \hfill
    \includegraphics[width=0.48\textwidth]{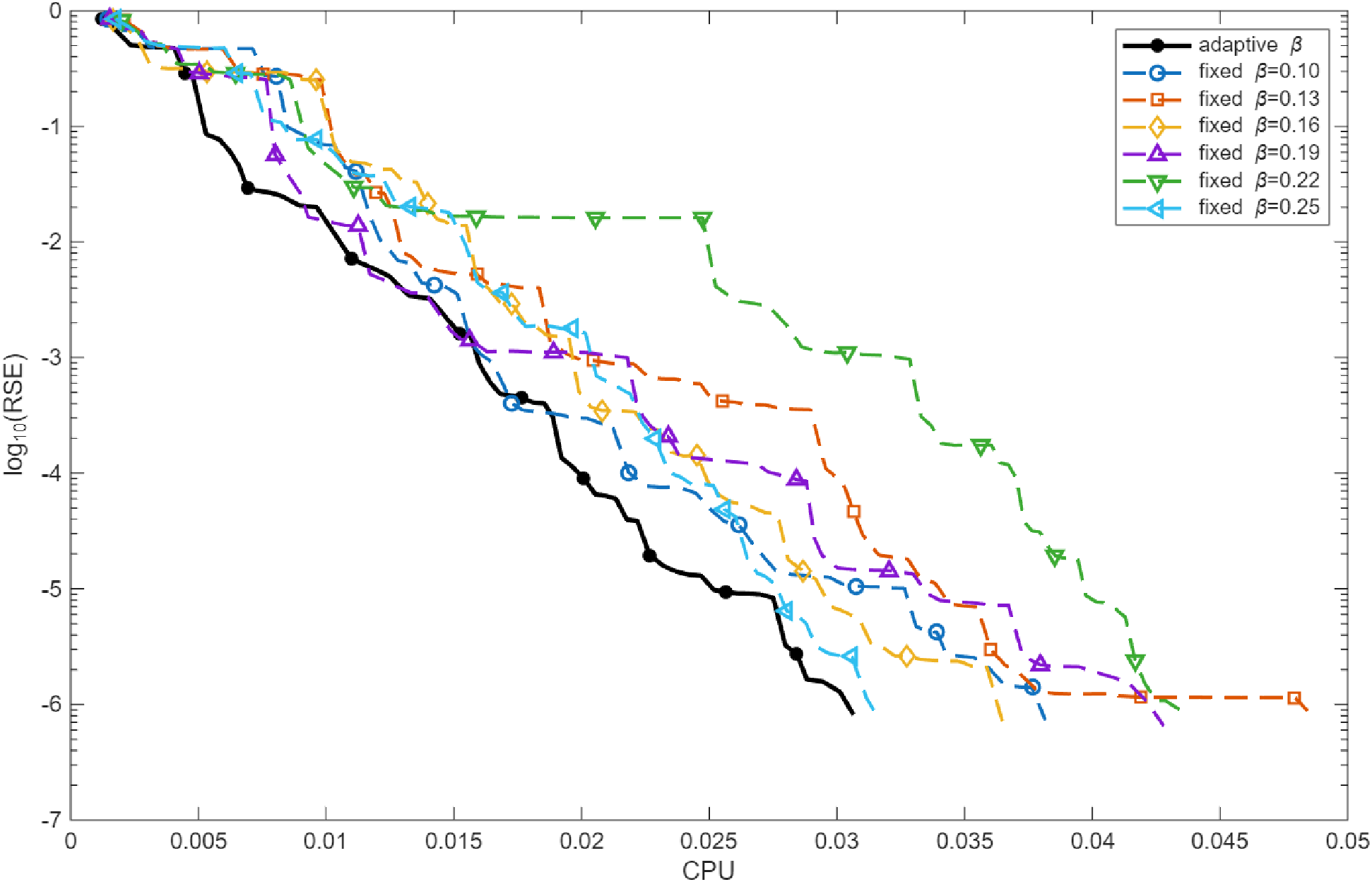}

    \caption{Relative solution error versus iteration count (left) and CPU time (right) under different values of $\beta$.}
    \label{fig:fix_beta}
\end{figure}

The convergence curves for Example~\ref{ex:6.0} under various values of $\beta$ are shown in Fig.~\ref{fig:fix_beta}. It is observed that the adaptive-$\beta$ strategy eliminates the need to manually preselect the momentum parameter, thereby saving the computational cost associated with parameter tuning. Furthermore,  this strategy achieves iteration counts and CPU time that are superior to or close to the best results among the fixed-$\beta$ cases.

\begin{example}\label{ex:6.4}
Video restoration experiments are conducted using the three methods. Specifically:

(1) $\mathcal{X} \in \mathbb{R}^{320\times 240\times 75}$ corresponds to the video
\texttt{v\_tennis\_01\_01.avi}, with the measurement tensor $\mathcal{A}\in \mathbb{R}^{480\times 320\times 75}$
generated from a standard normal distribution;

(2) $\mathcal{X}\in \mathbb{R}^{120\times 160\times 120}$ corresponds to the video
\texttt{traffic.avi} from Matlab, whose first forty frames are chosen. The corresponding measurement tensor $\mathcal{A}\in \mathbb{R}^{180\times 120\times 120}$ is
generated from a standard normal distribution.
\end{example}

In Example~\ref{ex:6.4}, the three methods are evaluated on video restoration problems. Peak signal-to-noise ratio (PSNR), structural similarity index (SSIM) and the minimum $\log_{10}(\mathrm{RSE})$ within a fixed number of iterations are adopted as evaluation standards.

\begin{figure}[!htbp]
    \centering
    \includegraphics[width=0.9\textwidth]{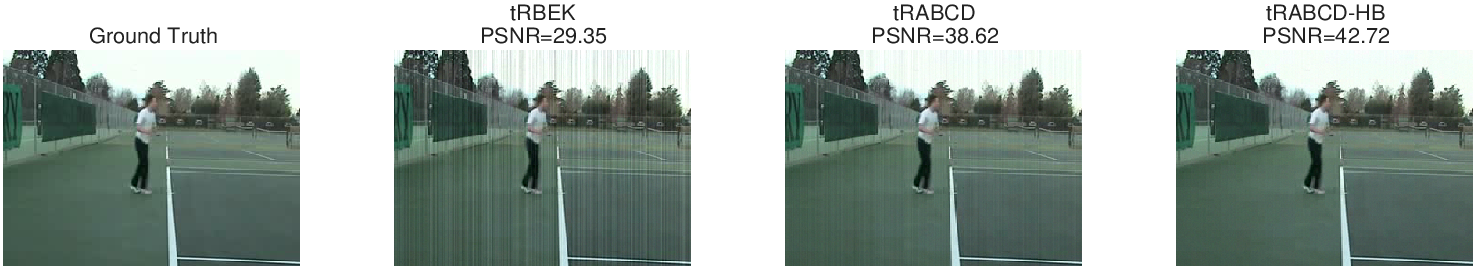}

    \vspace{1em}
	\includegraphics[width=0.9\textwidth]{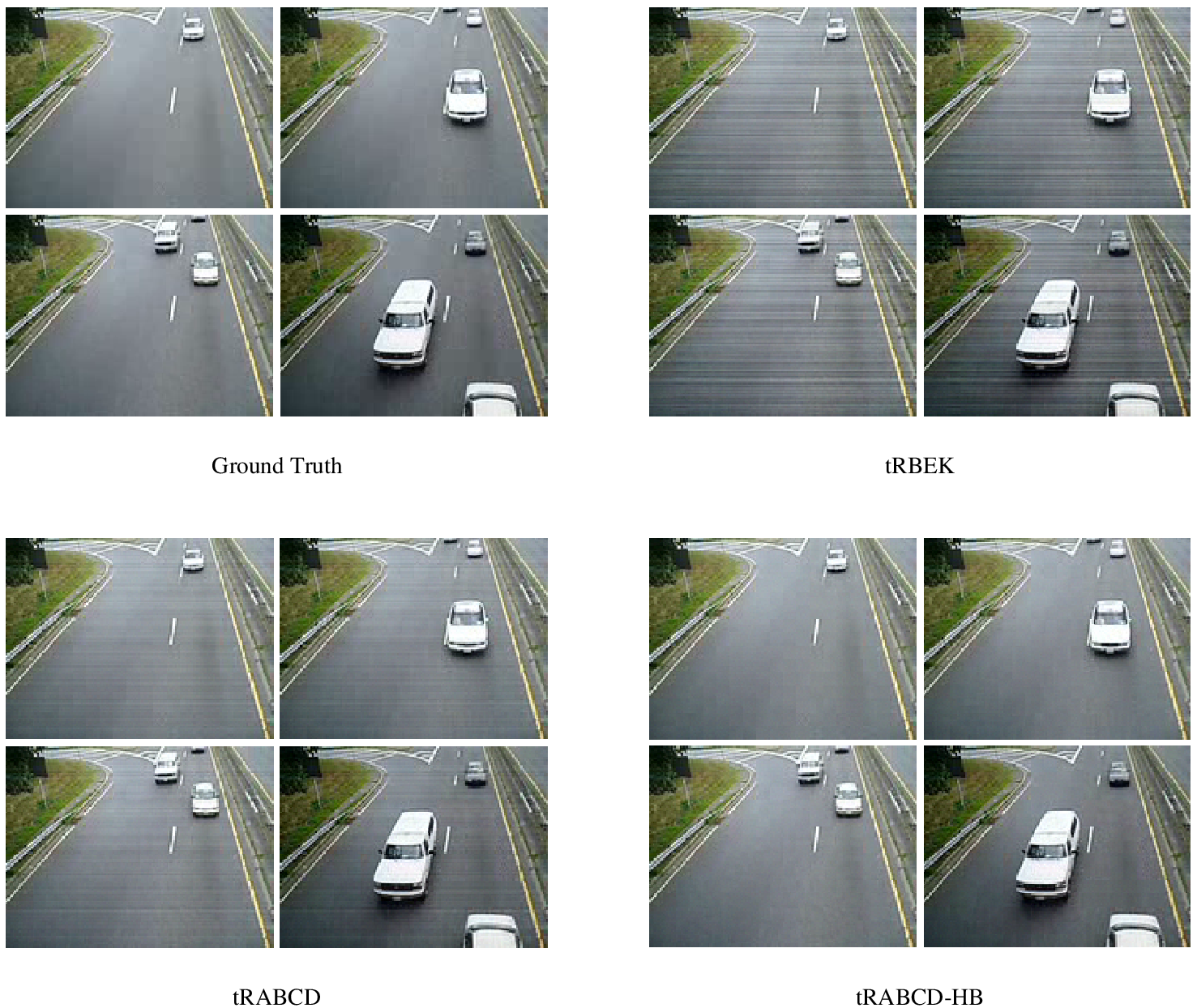}

    \caption{Original images, reconstructed images of video \texttt{v\_tennis\_01\_01.avi} (above) and video \texttt{traffic.avi} (below).}
    \label{fig:6.6}
\end{figure}

The visual recovery results for the first frame of the video in (1) and frames 10, 20, 30 and 40 of the video in (2) are shown in Fig.~\ref{fig:6.6}. It is seen from Fig.~\ref{fig:6.6} that the tensor randomized average block coordinate descent method with heavy-ball momentum produces visually superior reconstructed images compared with the other two methods under the same number of iterations.

\begin{figure}[!htbp]
    \centering
    \begin{minipage}{0.48\textwidth}
        \centering
        \includegraphics[width=\textwidth]{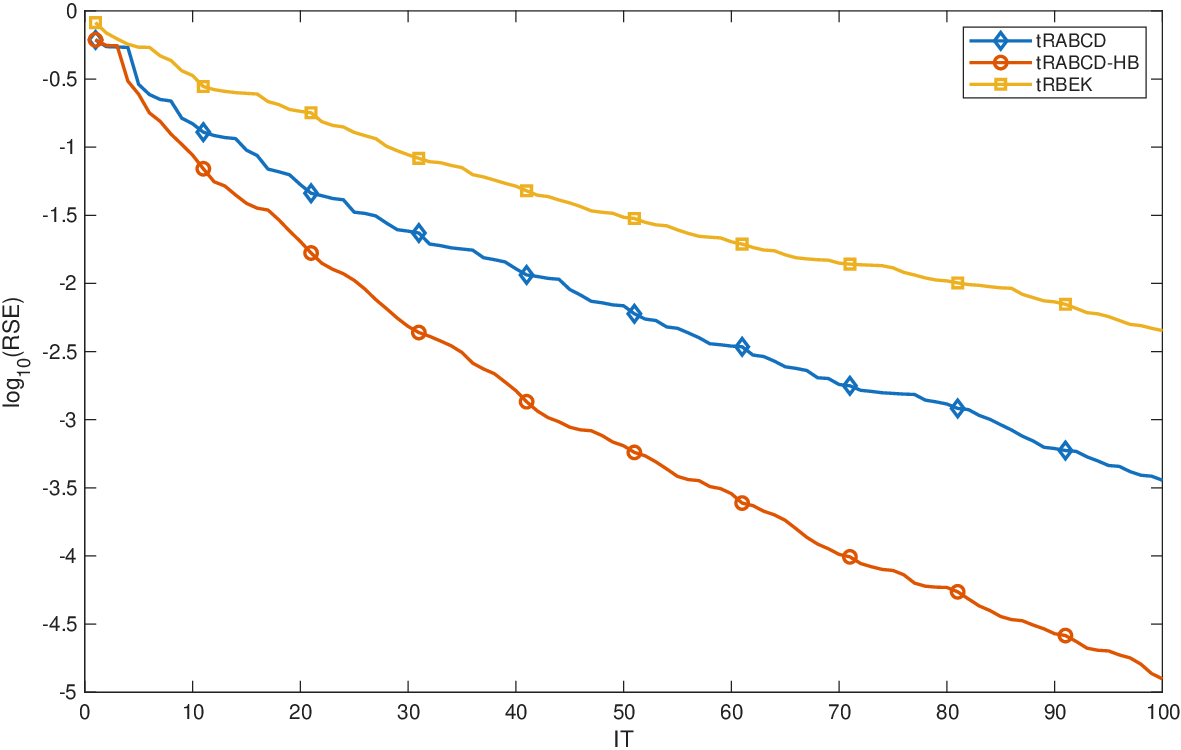}
    \end{minipage}\hfill
    \begin{minipage}{0.48\textwidth}
        \centering
        \includegraphics[width=\textwidth]{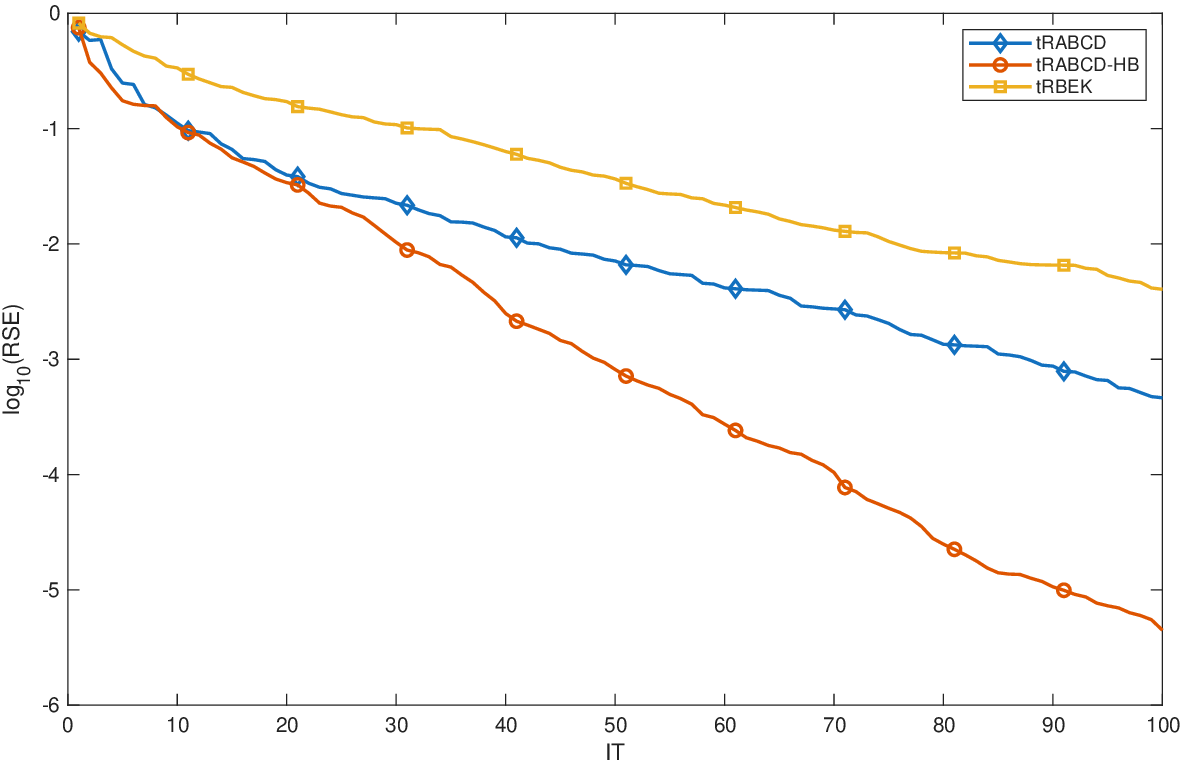}
    \end{minipage}
    \caption{Convergence curves for video \texttt{v\_tennis\_01\_01.avi} (left) and video \texttt{traffic.avi} (right).}
    \label{fig:6.7}
\end{figure}

\begin{table}[!htbp]
	\centering
	\caption{Numerical results for Example~\ref{ex:6.4}.}
	\label{tab:6.4}
	\begin{tabular}{l ccc ccc}
		\toprule
		\multirow{2}{*}{Method} & \multicolumn{3}{c}{\texttt{v\_tennis\_01\_01.avi}} & \multicolumn{3}{c}{\texttt{traffic.avi}} \\
		\cmidrule(lr){2-4} \cmidrule(lr){5-7}
		& PSNR & SSIM & $\log_{10}(\mathrm{RSE})$ & PSNR & SSIM & $\log_{10}(\mathrm{RSE})$ \\
		\midrule
		tRBEK     & 29.3509 & 0.8369 & -2.3463 & 29.3539 & 0.8231 & -2.3935 \\
		tRABCD    & 38.6214 & 0.9688 & -3.4441 & 37.3269 & 0.9591 & -3.3348 \\
		tRABCD-HB & \textbf{42.7222} & \textbf{0.9864} & \textbf{-4.9020} & \textbf{42.0061} & \textbf{0.9826} & \textbf{-5.3494} \\
		\bottomrule
	\end{tabular}
\end{table}

Numerical results of the three methods for the recovery of the two videos are presented in Fig.~\ref{fig:6.7} and Table~\ref{tab:6.4}. It is seen from Fig.~\ref{fig:6.7} that the convergence curves indicate the momentum mechanism helps maintain a consistent decrease in relative error during early iterations. Moreover, it is observed from Table~\ref{tab:6.4} that the tensor randomized average block coordinate descent method with heavy-ball momentum delivers the best overall performance for both test videos. Compared to the other two methods, it consistently achieves the highest PSNR and SSIM and the lowest $\log_{10}(\mathrm{RSE})$.

\section{Conclusion}
\label{sec:con}
A tensor randomized average block coordinate descent method with heavy-ball momentum is proposed for solving tensor least squares problem under $t$-product.
An adaptive step-size strategy without pseudoinverse computations is firstly introduced to construct an efficient baseline method, which is then enhanced by incorporating heavy-ball momentum. At each iteration, a two-dimensional subspace formed by the current randomized block direction and historical residual difference is exploited to determine the optimal step size and momentum parameters via orthogonal projection, yielding a tighter convergence bound with little additional cost. Theoretical analysis establishes the convergence rate and reveals an improved contraction effect over the baseline methods. Numerical experiments also confirm the effectiveness of the tensor randomized average block coordinate descent method with heavy-ball momentum, demonstrating its superior computational efficiency and stability in video restoration tasks.

\bibliographystyle{plain}
\bibliography{myref}

\end{document}